\documentclass[11pt]{amsart}

\usepackage[latin1]{inputenc}
\usepackage[english]{babel}
\usepackage{amsmath}
\usepackage[T1]{fontenc}
\usepackage{amsfonts}
\usepackage{amssymb}
\usepackage{amsthm}
\usepackage{fancyhdr}
\usepackage{amscd}
\usepackage{latexsym}
\usepackage{psfrag}
\usepackage{graphicx}
\usepackage[all]{xy}
\usepackage{booktabs}
\usepackage{multicol}
\usepackage{array}
\newcolumntype{C}{>{$}c<{$}}
\newcolumntype{L}{>{$}l<{$}}
\newcolumntype{R}{>{$}r<{$}}

\newcommand{\dual}{\vee}

\newcommand{\bbP}{\mathbb{P}}
\newcommand{\bbC}{\mathbb{C}}

\newcommand{\bbGr}{\mathbb{G}r}

\newcommand{\sO}{\mathcal{O}}

\newcommand{\sN}{\mathcal{N}}

\newcommand{\isom}{\cong}

\newcommand{\into}{\rightarrow}

\DeclareMathOperator{\slm}{SL}

\DeclareMathOperator{\rank}{rank}

\DeclareMathOperator{\codim}{codim}

\DeclareMathOperator{\Hilb}{Hilb}

\DeclareMathOperator{\Cat}{Cat}
\DeclareMathOperator{\sym}{Sym}

\theoremstyle{plain}
\newtheorem{teo}{Theorem}[section]
\newtheorem{propo}[teo]{Proposition}

\newtheorem{cor}[teo]{Corollary}
\newtheorem{obs}[teo]{Observation}

\newtheorem{lem}[teo]{Lemma}

\theoremstyle{definition}
\newtheorem{definition}[teo]{Definition}

\theoremstyle{remark}

\newtheorem{fn*}{Formule Notevoli}

\begin{document}
\title{Apolar Ideal and Normal Bundle of Rational Curves}
\author{Alessandro Bernardi}
\address{Dipartimento di Matematica, University of Florence}
\email{bernardi@math.unifi.it}

\date{\today}
\begin{abstract}
As in our previous work \cite{bernardi1} we address the problem to determine the splitting of the normal bundle of rational curves. With apolarity theory we are able to characterize some particular  subvarieties in some Hilbert scheme of
rational curves, defined by the splitting type of the normal bundle
and the restricted tangent bundle.

\end{abstract}
\maketitle

\section{Introduction}
In this work we address the problem of classifying the rational
curves $C\subset \bbP^m$ of degree $ n$, $n\geq m$, by the splitting
type of their normal bundle $N_{C;\bbP^m}$ or/and their restricted
tangent bundle $T\bbP^m|_C$. This problem was addressed by several
authors (see
\cite{Eisenbud-VandeVen1},\cite{Eisenbud-VandeVen2},\cite{Ghione-Sacchiero},\cite{Ghione},\cite{Miret},\cite{Ramella1},\cite{Ramella2},\cite{verdier},\cite{GHI}).
Our approach consists in a projective point of view, by considering
 our degree $n$ curves as projections of the rational normal curve $C_n\subset
\bbP^n$ from a linear space $L=\bbP^{k-1}$. The projected curves lie
in a projective space of dimension $m=(n-k)$. We point out that we
are interested in the case when $C$ has only ordinary singularities,
as in the work of Ghione and Sacchiero (\cite{Ghione-Sacchiero}). 

We define the scheme $H^{m,n}$ as the component of the Hilbert scheme $\Hilb_n\bbP^m$ of arithmetic genus zero curves of degree $n$ in $\bbP^m$ containing the smooth curves as an open subset.

We denote with $\sN^n_{n-k}(n_1,...,n_{n-k-1})$ the subscheme of curves such that the splitting type of normal bundle is $(n_1,...,n_{n-k-1})$.

One advantage in working directly on the Grassmannian $Gr(\bbP^{k-1}, \bbP^n)$ is that the number of irreducible components and the codimension of the varieties $N^n_{n-k}(n_1,...,n_{n-k-1})$ defined as $\sN^n_{n-k}(n_1,...,n_{n-k-1})/SL(n-k+1)$ (parameterizing subspaces $L$ such that the curve obtained by projecting from $L$
has normal bundle isomorphic to $\bigoplus^{n-k-1}_{i=1}\sO_{\bbP^1}(n_i)$ ) remain the same as those of $\sN^n_{n-k}(n_1,...,n_{n-k-1})$. Then we can study directly the basic structures of these subvarieties in the Grassmannian. In the same way we define the varieties $T^n_{n-k}(t_1,...t_{n-k})$ with respect to the restricted tangent bundle.

In the results in  \cite{bernardi1}, we have studied some Hilbert
scheme of rational curves with fixed splitting type of the normal
bundle and the restricted tangent bundle and we have characterized
them via the individuation of some particular family of multisecant
projective spaces containing $L$. In this work, instead, we
characterize some particular subvarieties in some Hilbert scheme of
rational curves, defined by the splitting type of the normal bundle
or the restricted tangent bundle. We obtain our main results
(Theorems \ref{mainresult}, \ref{mainresultTG}) by using apolarity
theory. In particular we will show that:

For $n-1\leq 3k\leq 3(n-3$), if the center of projection
$L\isom\bbP^{k-1}$ is contained in a  $\bbP^{k+1}$, $(k+2)-$secant
to the rational normal curve $C_n \subset \bbP^n$, then we have:
$$
N_{\pi_L(C_n);\bbP^{n-k}}\isom \sO(n+2)^{n-k-2}\oplus \sO(n+1+2k) .
$$

For $2k<n-1$, if the center of projection $L\isom\bbP^{k-1}$ is
contained in a $\bbP^{k}$, $(k+1)-$secant to the rational normal
curve $C_n \subset \bbP^n$, then we have:
$$
T\bbP^{n-k}|_{\pi_{k}(C_n)}\isom \sO(n+1)^{n-k-1}\oplus\sO(n+1+k).
$$
One of the main reason of interest in this work is represented by
the study of the interplay between apolarity theory and the
splitting type of the normal and conormal bundle (see also
\cite{bernardi1}).
\section{Rational Curves of degree $n$ in codimension $k$}
We describe here the main steps of our approach, referring to
 \cite{bernardi1} for more details. Let $C_n=\nu_n(\bbP^1)\subset \bbP^n$ be the
rational normal curve which is the image of the Veronese map
$\nu_n:\bbP^1\into \bbP^n$. Let $\pi_L(C_n)$ be the rational curve
obtained from $C_n$ by projection from a $k$-dimensional linear
subspace $L\subset \bbP^n$ on $\bbP^{n-k}\subset \bbP^n$;  we will
suppose that $\pi_L(C_n)$ has only ordinary singularities.

%
%

 Let $J(\nu_n)$ be the Jacobian matrix of $\nu_n$:
$$
J(\nu_n)=
\left(
\begin{array}{ccccc}
ns^{n-1} & \dots & t^{n-1} & 0 \\
0 & s^{n-1} & \dots & nt^{n-1}
\end{array}
\right).
$$
We will indicate with $V$ a 2-dimensional complex vector space, therefore we can write down in more invariant way:
\begin{equation}\label{diagr:normCODIMK}
\xymatrix{
&&0\ar[d]&0\ar[d]\\
& & \bbC^{k}\otimes\sO_{C_n} \ar[d]^P \ar[r]^{\isom}& \bbC^{k}\otimes\sO_{C_n} \ar[d]^{(\mathcal{N}^L_{n,k})^t} &  \\
0 \ar[r] & V\otimes\sO_{\bbP^1}(-n+1) \ar[r]^{J(\nu_n)} & \sym^nV\otimes\sO_{C_n} \ar[d] \ar[r]^{Syz(J(\nu_n))} & \sym^{n-2}V\otimes \sO_{C_n}(2) \ar[d] \ar[r] & 0 \\
0 \ar[r] & V\otimes\sO_{\bbP^1}(-n+1) \ar[r]^{J(\pi_{n-3}\circ\nu_n)} & \frac{\sym^nV}{\bbC^{k}}\otimes\sO_{\pi_L(C_n)} \ar[d]\ar[r] & N'_{\pi_L(C_n);\bbP^{n-k}}(-n) \ar[d] \ar[r] & 0\\
&&0&0.
}
\end{equation}
where the map $(\mathcal{N}^L_{n,k})^t$ is given by:
$$
(\mathcal{N}^L_{n,k})^t=
Syz(J(\nu_n))\cdot \left(
\begin{array}{ccccc}
a^1_0 & \dots & a^{k}_0 \\
\vdots & \ddots & \vdots \\
a^1_n & \dots & a^{k}_n
\end{array}
\right)=
$$
$$
\left(
\begin{array}{ccccccc}
a^1_0 t^2-2a^1_1 ts+a^1_2 s^2 & a^1_1 t^2-2a^1_2 ts+a^1_3 s^2 & \dots & a^1_{n-2} t^2-2a^1_{n-1} ts+a^1_n s^2\\
\vdots & \vdots & \ddots & \vdots \\
a^{k}_0 t^2-2a^{k}_1 ts+a^{k}_2 s^2 & a^{k}_1 t^2-2a^{k}_2 ts+a^{k}_3 s^2 & \dots & a^{k}_{n-2} t^2-2a^{k}_{n-1} ts+a^{k}_n s^2
\end{array}
\right)^t,
$$
 a $(n-1)\times k$ matrix. We stress out that the first row of the above diagram is $\slm(2)$ invariant, but the second one is not.
\begin{obs}\label{obs:singularities}
We can observe that if $\pi_L(C_n)$ has only ordinary singularities, then the map of differential is surjective (see \cite{Ghione-Sacchiero}), so ${N'}^{\vee}_{\pi_L(C_n);\bbP^{n-k}}$ is a vector bundle. We will consider only cases with ordinary singularities, so
we will indicate ${N'}_{\pi_L(C_n);\bbP^{n-k}}$ as $N_{\pi_L(C_n);\bbP^{n-k}}$.
\end{obs}
We can obtain as in the case of the normal bundle of rational
curves, the following exact sequence for the restricted tangent
bundle :
\begin{equation}\label{succ:T}
\xymatrix{
0 \ar[r] & (T\bbP^{n-k}|_{\pi_L(C_n)})^{\vee}(n+1) \ar[r] & \sym^{n-1}V\otimes\sO_{C_n}^n \ar[r]^{\mathcal{T}^L_{n,k}}& \bbC^k\otimes\sO_{C_n}(1) \ar[r] & 0},
\end{equation}

where we have indicated with $T^L_{n,k}$ the $2k\times n$ matrix:
$$
T^L_{n,k}=
\left(
\begin{array}{ccccc}
a^1_0 & \dots & a^1_{n-1}\\
-a^1_1 & \dots & -a^1_{n}\\
\vdots & \ddots & \vdots \\
a^{k}_0 & \dots & a^{k}_{n-1}\\
 -a^{k}_1 & \dots & -a^{k}_n
\end{array}
\right).
$$

We refer to  \cite{bernardi1} for the main results about apolarity
and the Waring's Problem. Moreover we need the following  results (see \cite{iarrobinokanev}):
\begin{definition}
Let $f\in S_n$ be a binary form of degree $n$. Let $L_1,...,L_m$ be a linear forms. A representation of $f$ as a sum:
\begin{equation}\label{GAD}
f=G_1L^{n-g_1+1}_1+...+G_sL^{n-g_m+1}_m,
\end{equation}
where $G_i\in S_{g_i-1}$, is called a generalized additive decomposition (GAD) of $f$. A GAD is called normalized if no pair $L_{\alpha},L_{\beta}$ is proportional to each other and none of the $G_i$ is divisible by $L_i$. Its length is by definition $\sum^{m}_{i=1}g_i$.

If all $g_i=1$ we obtain the classical additive decomposition:
\begin{equation}
f=c_1L^n_1+...+c_sL^n_s, 
\end{equation}
with $c_i\in \bbC$.

The length of a binary form $f$ is the minimum length of a GAD of $f$, we denote it by $l(f)$.
\end{definition}
\begin{lem}[\cite{iarrobinokanev}]\label{GADLEMMA}
Let $\phi=\prod^m_{i=1}(b_i\partial_0-a_i\partial_1)^{g_i}$ be a prime decomposition of a nonzero form in $T_s$. Let $L_i=a_ix_0+b_ix_1$. Then a form $f\in S_n$ with $n\geq s$ has a GAD as \ref{GAD} if and only if $\phi$ is apolar to $f$. If all roots of $\phi$ are simple, then this is an additive decomposition.
\end{lem}
\begin{lem}[\cite{iarrobinokanev}]
Let $n=2t$ or $n=2t+1$, let $f\in S_n$. Then $l(f)\leq t+1$. If $I_f=Ann(f)$ is the ideal of forms apolar to $f$, then $l(f)$ equals the order $d$ (i.e. the initial degree) of the graded ideal $I_f=(I_f)_d+(I_f)_{d+1}+...$.
\end{lem}

\begin{lem}[Jordan's Lemma]
Suppose that the linear forms $L_i,i=1,...,m$ are not proportional to each other and:
$$
0=G_1L^{n-g_1+1}_1+...+G_sL^{n-g_m+1}_m,
$$
with $\sum^m_{i=1}g_i\leq n+1$. Then $G_i=0$ for every $i$.
\end{lem}
\begin{propo}[Uniqueness of GAD]
Suppose $n=2t$ or $n=2t+1$. Let:
$$
f=G_1L^{n-g_1+1}_1+...+G_mL^{n-g_m+1}_m,
$$
be a normalized $GAD$ of $f\in S_n$ of length $s=\sum^m_{i=1}g_i\leq t+1$. Then $f$ has no other $GAD$ of length $\leq n+1-s$ and $l(f)=s$. In particular if $s\leq t$ or if $s=t+1, n=2t+1$ (equivalently $2s\leq n+1$), then the above is the unique normalized GAD of $f$ having length $\leq t+1$.
\end{propo}
\begin{definition}
Let $f\in S_n$ and let $2l(f)\leq n+1$. Then the unique normalized GAD of length $s=l(f)$ is called the canonical form of $f$.
\end{definition}
\begin{teo}[Sylvester]
\begin{itemize}
\item[i)] For odd $n=2t+1$, the general $f\in S_n$ has a unique decomposition as a sum of $t+1$ $n-$th powers of linear forms.
\item[ii)] For even $n=2d$, the general $f\in S_n$ has infinitely many decompositions as a sum of $t+1$ $n-$th powers of linear forms.
\end{itemize}
\end{teo}
\begin{teo}\label{rankcat}
Let $n=2t$ or $2t+1$, let $f\in S_n$.
\begin{itemize}
\item[i)] Let $s=\rank \Cat_f(n-t,t;2)$. Then $l(f)=s$. If $2s\leq n+1$, then $f$ has a unique generalized additive decomposition of length $s$. and no other GADs of length $\leq t+1$.
\item[ii)] For every pair of integers $s,e$ with $1\leq s\leq e\leq n-e+1$, if $l(f)=s$, then $l(f)=s=\rank \Cat_f(n-e,e;2)$.
\end{itemize}
\end{teo}
\begin{teo}[\cite{iarrobinokanev}]\label{completeintersection}
Let $f$ be a binary form of degree $n=2t$ or $2t+1$ and $I_f=Ann(f)$ be the ideal of forms apolar to $f$. Let $A_f=S/I_f$ be the associated Gorenstein Artin algebra. Let $s=\max \{\dim(A_f)_i\}$. Then: 
\begin{itemize}
\item[i.] $s=l(f)$ and the Hilbert function of $A_f$ satisfies:
$$
H(A_f)=(1,2,...,s-1,\begin{array}{c}s-1\\s\end{array},s,...,\begin{array}{c}n-s+1\\s\end{array},s-1,...,2,1);
$$
\item[ii.] Suppose $2s\leq n+1$. Then $\dim I_s=1$, $I_s=<\alpha>$ and for every integer $v$ with $s\leq v\leq n-s+1$ one has $I_v=S_{v-s}\circ \alpha$; 
\item[iii.] The apolar ideal $I_f$ is generated by two homogeneous polynomials $\alpha\in (I_f)_s$ and $\beta \in (I_f)_{n+2-s}$.\\ Equivalently the ring $A_f$ is a complete intersection of generator degrees $s,n+2-s$. The two polynomials above have no common zeros.
\end{itemize}
\end{teo}
 In this contest our principal tool will be the Apolarity Lemma (see \cite{iarrobinokanev}):
\begin{lem}[Apolarity Lemma ]\label{apolarity}

Let $p_1,...,p_s\in \bbP^r$, let $L_i=L_{p_i}$, let $P=\{[p_1],...,[p_s]\}\subset \bbP^{r}$ and let $\mathcal{I}_P$ be the homogeneous ideal in $T$ of polynomials vanishing on $P$. Then:
\begin{itemize}
\item[i)] For every $\phi \in R_e$: 
$$
\phi \circ (L^{[N]}_1+...+L^{[N]}_s)=\phi (p_1)L^{[N-E]}_1+...+\phi (p_s)L^{[N-E]}_s.
$$
\item[ii)] With respect to the contraction paring $T_n\times S_n\into \bbC$ one has:
$$
((\mathcal{I}_P)_n)^{\perp}=<L^{[N]}_1,...,L^{[N]}_s>.
$$
\item[iii)] The points $[p_1],...,[p_s]\subset \bbP^{r}$ impose independent conditions on the linear system $|\sO_{\bbP^{n}}(j)|$ if and only if $L^{[N]}_1,...,L^{[N]}_s$ are linearly independent.
\item[iv)] Suppose $s\leq \dim_{\bbC} T_{n-e}$ and the linear forms $L^{[N]}_1,...,L^{[N]}_s$ have the property that the corresponding set $P$ imposes independent conditions on the linear system $|\sO_{\bbP^{r}}(n-e)|$. Let $f=L^{[N]}_1+...+L^{[N]}_s$. Then we have for the apolar forms to $f$ of degree $e$ the equality:
$$
Ann(f)_e=(\mathcal{I}_P)_e.
$$
\end{itemize}  
\end{lem}
\subsection{Normal Bundle of Rational Curves in $\bbP^{n-1}$}
It is easy to prove the following result:
\begin{teo}
Let $C_n\subset \bbP^n$ be the rational normal curve of degree $n$.
For any $p\in \bbP^n$, the rational curve $\pi_p(C_n)\subset
\bbP^{n-1}$ has normal bundle $N_{\pi_p(C_n);\bbP^{n-1}}
=\sO(n+2)^{n-4}\oplus\sO(n+3)^2$ if and only if $p$ is not on a
secant (or tangent) line to $C_n$. This is equivalent to saying that
$\pi_p(C_n)$ is smooth.
\end{teo}

\medskip

\subsection{Normal Bundle of Rational Curves in $\bbP^{n-2}$}

By diagram (\ref{diagr:normCODIMK}) in case of codimension $2$ 
since the map  $\sO_{C_n}(2)^{n-1}\into N_{\pi_{L}(C_n);\bbP^{n-2}}(-n)$ is surjective, it follows that $2\leq {n'}_0\leq ...\leq {n'}_{n-4}$ and ${n'}_0+ ...+ {n'}_{n-4}=2n-2$, where $N_{\pi_{L}(C_n);\bbP^{n-2}}(-n)=\sO({n'}_0)\oplus ...\oplus\sO({n'}_{n-4})$. But we can write ${n'}_i={n''}_i+2$, so ${n''}_0+ ...+ {n''}_{n-4}=4$.

Therefore we have ${n'}_0= ...= {n'}_{n-8}=2$, we can write:
$$
N_{\pi_2(C_n);\bbP^{n-2}}=\sO(n+2)^{n-7}\oplus \mathcal{F};
$$
where $\mathcal{F}$ is a rank 4 vector bundle on $\bbP^1$. Hence we must study the splitting of $\mathcal{F}$, if we indicate with $\mathcal{F}=\sO(f_0)\oplus...\oplus\sO(f_3)$, where $f_0+...+f_3=4n+14$. 
Therefore it is one of the following cases:

\begin{enumerate}
\item $\mathcal{F}=\sO(n+5)^4=:\mathcal{F}_1$;
\item $\mathcal{F}=\sO(n+4)\oplus\sO(n+5)^2\oplus\sO(n+6)=:\mathcal{F}_2$;
\item $\mathcal{F}=\sO(n+4)^2\oplus\sO(n+6)^2=:\mathcal{F}_3$;
\item $\mathcal{F}=\sO(n+4)^2\oplus\sO(n+5)\oplus\sO(n+7)=:\mathcal{F}_4$;
\item $\mathcal{F}=\sO(n+4)^3\oplus\sO(n+8)=:\mathcal{F}_5$;
\end{enumerate}
no other case can occur if the projection has only ordinary
singularities by Observation.\ref{obs:singularities}.

Since the codimension is $2$, by dualizing the last exact column of
(\ref{diagr:normCODIMK}) and tensorizing with $\sO_{\bbP^1}(2)$, we
get:
\begin{displaymath}
\xymatrix{
0 \ar[r] & N^{\dual}_{\pi_2(C_n);\bbP^{n-2}}(n+2)  \ar[r] & \sO_{C_n}^{n-1} \ar[r]^{\mathcal{N}^L_{n,2}} & \sO^2_{C_n}(2) \ar[r] & 0, }
\end{displaymath}
and we have $deg(N^{\dual}_{\pi_L(C_n);\bbP^{n-2}}(n+2))=-2$. But
$N^{\dual}_{\pi_L(C_n);\bbP^{n-2}}(n+2)=\sO^{n-7}\oplus\mathcal{F}^{\dual}(n+2)$,
so
$h^0(N^{\dual}_{\pi_L(C_n);\bbP^{n-2}}(n+2))=n-7+h^0(\mathcal{F}^{\dual}(n+2))$.
Therefore we have that $0\leq h^0(\mathcal{F}^{\dual}(n+2)) \leq 3$,
where:
\begin{enumerate}
\item[a)] $h^0(\mathcal{F}^{\dual}(n+2))=0 \Leftrightarrow \mathcal{F}=\mathcal{F}_1$;
\item[b)] $h^0(\mathcal{F}^{\dual}(n+2))=1 \Leftrightarrow \mathcal{F}=\mathcal{F}_2 $;
\item[c)] $h^0(\mathcal{F}^{\dual}(n+2))=2 \Leftrightarrow \mathcal{F}=\mathcal{F}_3\;\; or\;\;\mathcal{F}_4$;
\item[d)] $h^0(\mathcal{F}^{\dual}(n+2))=3 \Leftrightarrow \mathcal{F}=\mathcal{F}_5 $.
\end{enumerate}

So we have the following cases:
\begin{enumerate}
\item[A)] $\rank(N^L_{n,2})=6\Leftrightarrow h^0(\mathcal{F}^{\dual}(2))=0 \Leftrightarrow \mathcal{F}=\mathcal{F}_1  $;
\item[B)] $\rank(N^L_{n,2})= 5 \Leftrightarrow h^0(\mathcal{F}^{\dual}(2))=1 \Leftrightarrow \mathcal{F}=\mathcal{F}_2 $;
\item[C)] $\rank(N^L_{n,2})=4\Leftrightarrow h^0(\mathcal{F}^{\dual}(2))=2 \Leftrightarrow \mathcal{F}=\mathcal{F}_3\;\; or\;\;\mathcal{F}_4 $;
\item[D)] $\rank(N^L_{n,2})=3\Leftrightarrow h^0(\mathcal{F}^{\dual}(2))=3 \Leftrightarrow \mathcal{F}=\mathcal{F}_5 $.
\end{enumerate}
If we consider the following exact sequence:
\begin{displaymath}
\xymatrix{
0 \ar[r] & N^{\dual}_{\pi_L(C_n);\bbP^{n-2}}(n+3)  \ar[r] & \sO_{C_n}^{n-1}(1) \ar[r]^{\mathcal{N}^L_{n,2}} & \sO^2_{C_n}(3) \ar[r] & 0, }
\end{displaymath}
we have that:
\begin{enumerate}
\item[C1)] $\rank(N^L_{n,2}(1))=3\Leftrightarrow h^0(\mathcal{F}^{\dual}(3))=8 \Leftrightarrow \mathcal{F}=\mathcal{F}_3$;
\item[C2)] $\rank(N^L_{n,2}(1))\leq4\Leftrightarrow h^0(\mathcal{F}^{\dual}(3))=7 \Leftrightarrow \mathcal{F}=\mathcal{F}_4$.

\end{enumerate}

\normalsize

\begin{propo}\label{apolaritycorollaryN}
If there exist two points $q_1,q_2\in L$, each of them belonging to
a different $3-$secant $\bbP^2$, then:
$$
\ker(N^L_{n,2})=(\mathcal{I}_{D_1}\cap\mathcal{I}_{D_2})_{n-2}=(\mathcal{I}_{D_1\cup D_2})_{n-2},
$$
where $D_i$ is the set of points in $\bbP^1$ which corresponds to
the linear forms in the additive decomposition of $f_i$ (the binary form corresponding to $q_i$).
\end{propo}

\subsubsection{Case $\rank N^L_{n,2}=3$}

\begin{lem}
The splitting type of the normal bundle $N_{\pi_{L}(C_n),\bbP^{n-2}}$ is:
$$
((n+2)^{n-4},n+6)
$$
if and only if $\rank(N^L_{n,2})=3$.
\end{lem}
\begin{teo}[Case Rank 3]
If the projection line $L$ belongs to some $3-$secant $\bbP^2$, but it is not a secant line, then the splitting type of the normal bundle $N_{\pi_{L}(C_n),\bbP^{n-2}}$ is:
$$
((n+2)^{n-4},n+6).
$$

\end{teo}
\begin{proof}
If $L$ belongs to a $3-$secant $\bbP^2$, then there exist two points
$q_1,q_2\in L$ and their corresponding binary forms $f_1,f_2$ of degree $n$  such that $Ann(f_1)=(\alpha,\beta_1)$ and
$Ann(f_2)=(\alpha,\beta_2)$ where $\alpha$ has only simple roots,
 $\deg(\alpha)=3$ and $\beta_1\neq\beta_2,\;\;
\deg(\beta_1)=\deg(\beta_2)=n-1$. So $\dim <\alpha>_{n-2}=n-4$ and
$<\alpha,\beta_1>_{n-2}=<\alpha,\beta_2>_{n-2}=<\alpha>_{n-2}$.
Therefore
$\rank(N^L_{n,2})=n-1-\dim(<\alpha,\beta_1>_{n-2}\cap<\alpha,\beta_2>_{n-2})=3$.

\end{proof}
\begin{cor}
The variety of lines $L$ such that $\pi_L$ gives a rational curve of
degree $n$ in $\bbP^{n-2}$ for which  the splitting type of the
normal bundle $N_{\pi_{L}(C_n),\bbP^{n-2}}$ is:
$$
((n+2)^{n-4},n+6)
$$
has an irreducible subvariety of codimension $(2n-7)$ in
$Gr(\bbP^1,\bbP^n)$, that is formed by the lines belonging to some
$3-$secant $\bbP^2$, but which are not secant lines.
\end{cor}
\subsubsection{Case $\rank N^L_{n,2}=4$}
\begin{lem}
The codimension in $Gr(\bbP^1,\bbP^n)$ of the variety of all lines
in $\bbP^n$ belonging to some 4-secant  $\bbP^3$ to the rational
normal curve in $\bbP^n$  is $2n-10$
\end{lem}
\begin{proof}
In fact we can consider the incidence variety $I=\{(L,\pi):L\in
Gr(\bbP^1,\bbP^n),\pi\in S, L\subset S\}$ where $S$ is the set of
all 4-secant $\bbP^3$. In the usual way we can compute the
codimension of the image of this incidence variety in
$Gr(\bbP^1,\bbP^n)$. That calculation is effective thanks to the
result of Chiantini and Ciliberto on the non-defectivity of the
Grassmannians of secant varieties of curves (see
\cite{chiantiniciliberto}).
\end{proof}
\begin{lem}
If the splitting type of the normal bundle $N_{\pi_{L}(C_n),\bbP^{n-2}}$ is:
$$
((n+2)^{n-5},(n+4)^2),
$$
then $\rank(N^L_{n,2})=4$. Moreover  the variety which parameterizes
the lines giving the above splitting has codimension $2(n-5)$ in
$Gr(\bbP^1,\bbP^n)$.
\end{lem}
\begin{teo}[Case Rank 4]
If the projection line $L$ belongs to some  $4-$secant $\bbP^3$, but
it does not belong to some $3-$secant $\bbP^2$, then the splitting
type of the normal bundle $N_{\pi_{L}(C_n),\bbP^{n-2}}$ is:
$$
((n+2)^{n-5},(n+4)^2).
$$

\end{teo}
\begin{proof}
If $L$ belongs to a  $4-$secant $\bbP^3$, then there exist two
points $q_1,q_2\in L$ and their corresponding binary forms $f_1,f_2$ of degree $n$ such that $Ann(f_1)=(\alpha,\beta_1)$ and
$Ann(f_2)=(\alpha,\beta_2)$ with $\alpha$ has only simple roots and
$\deg(\alpha)=4$ and $\beta_1\neq\beta_2,\;\;
\deg(\beta_1)=\deg(\beta_2)=n-2$. So $\dim <\alpha>_{n-2}=n-5$ and
$\dim(<\beta_1>_{n-2}\cap <\beta_2>_{n-2})=0$ otherwise
$\beta_1=\beta_2$ and $q_1=q_2$, but this is impossible. Therefore
$\rank(N^L_{n,2})=n-1-\dim(<\alpha,\beta_1>_{n-2}\cap<\alpha,\beta_2>_{n-2})=4$.

\end{proof}
\begin{cor}
The variety of lines $L\subset \bbP^n$ such that
$\pi_L(C_n)\subset\bbP^{n-2}$  has the splitting type of the normal
bundle $N_{\pi_{L}(C_n),\bbP^{n-2}}$:
$$
((n+2)^{n-5},(n+4)^2)
$$
has an irreducible component formed by the lines belonging to some
$4-$secant $\bbP^3$, but not contained in any $3-$secant $\bbP^2$.
\end{cor}

\subsubsection{Case $\rank N^L_{n,2}=5$}
\begin{lem}
The splitting type of the normal bundle $N_{\pi_{L}(C_n),\bbP^{n-2}}$ is:
$$
((n+2)^{n-6},(n+3)^2,n+4)
$$
if and only if $\rank(N^L_{n,2})=5$.
\end{lem}

\begin{obs}
If $L$ belongs to a $5-$secant $\bbP^4$, then there exist two points $q_1,q_2\in L$ and their corresponding binary forms $f_1,f_2$ of degree $n$  such that $Ann(f_1)=(\alpha,\beta_1)$ and $Ann(f_2)=(\alpha,\beta_2)$ with $\alpha$ has only simple roots  and $\deg(\alpha)=5$ and $\beta_1\neq\beta_2,\;\; \deg(\beta_1)=\deg(\beta_2)=n-3$. So $\dim<\alpha>_{n-2}=n-6$, so $\rank(N^L_{n,2})=n-1-\dim(<\alpha,\beta_1>_{n-2}\cap<\alpha,\beta_2>_{n-2})\leq 5$.
\end{obs}

\begin{obs}
If $\rank(N^L_{n,2})=5$, then or $L\subset \bbP^4$ which is $5-$secant to $C_n$ or  $L\subset \bbP^k$ which is $(k+1)-$secant. In the second case there exist two points $q_1,q_2\in L$ such that $Ann(f_1)=(\alpha,\beta_1)$ and $Ann(f_2)=(\alpha,\beta_2)$ with $\alpha$ has only simple roots  and $\deg(\alpha)=k+1$ and $\beta_1\neq\beta_2,\;\; \deg(\beta_1)=\deg(\beta_2)=n-k+1$.
\end{obs}

\subsection{Normal Bundle of Rational Curves in $\bbP^{n-3}$}
\normalsize
\subsubsection{Case $\rank N^L_{n,3}=6,5$}

\begin{obs}
If $L\isom \bbP^2$ is contained in a $6-$secant $\bbP^5$ to $C_n$, then there exist three binary forms of degree $n$ correspond to  three points $q_1,q_2,q_3\in L$
such that $Ann(f_1)=(\alpha,\beta_1),Ann(f_2)=(\alpha,\beta_2)$ and $Ann(f_3)=(\alpha,\beta_3)$ where $\alpha$ has only simple roots,
$\deg(\alpha)=6$ and $\beta_1\neq\beta_2\neq\beta_3,\;\; \deg(\beta_1)=\deg(\beta_2)=\deg(\beta_3)=n-4$. So $\dim<\alpha>_{n-2}=n-7$, so
$\rank(N^L_{n,2})=n-1-\dim(<\alpha,\beta_1>_{n-2}\cap<\alpha,\beta_2>_{n-2})\leq 6$.

\end{obs}

\begin{obs}
If $L\isom \bbP^2$ is contained in a a $5-$secant $\bbP^4$, then there exist three binary forms of degree $n$ correspond to three points $q_1,q_2,q_3\in L$ such that
$Ann(f_1)=(\alpha,\beta_1),Ann(f_2)=(\alpha,\beta_2)$ and $Ann(f_3)=(\alpha,\beta_3)$ where $\alpha$ has only simple roots,
 $\deg(\alpha)=5$ and $\beta_1\neq\beta_2\neq\beta_3,\;\; \deg(\beta_1)=\deg(\beta_2)=\deg(\beta_3)=n-3$.
 So $\dim<\alpha>_{n-2}=n-6$, and $\rank(N^L_{n,2})=n-1-\dim(<\alpha,\beta_1>_{n-2}\cap<\alpha,\beta_2>_{n-2})\leq 5$. 
\end{obs}
\subsubsection{Case $\rank N^L_{n,3}=4$}
\begin{propo}
If  $L\isom \bbP^2$  is contained in a  $4-$secant $\bbP^3$, but it
is not a $3-$secant $\bbP^2$, then the splitting type of the normal
bundle $N_{\pi_{L}(C_n),\bbP^{n-3}}$ is:
$$
((n+2)^{n-5},n+8).
$$
Such $L$'s form an irreducible component of codimension $(3n-13)$ of
the varieties of projection planes which give the splitting type
above .
\end{propo}
\begin{proof}
If $L\isom \bbP^2$, as centre of projection, belongs to a $4-$secant $\bbP^3$, then there exist three binary forms of degree $n$ correspond to three points $q_1,q_2,q_3\in L$ such that $Ann(f_1)=(\alpha,\beta_1)$, $Ann(f_2)=(\alpha,\beta_2)$ and $Ann(f_3)=(\alpha,\beta_3)$ with $\alpha$ has only simple roots  and $\deg(\alpha)=4$ and $\beta_1\neq\beta_2\neq \beta_3,\;\; \deg(\beta_1)=\deg(\beta_2)=\deg(\beta_3)=n-2$. So $\dim<\alpha>_{n-2}=n-5$, so $\rank(N^L_{n,3})=n-1-\dim(<\alpha,\beta_1>_{n-2}\cap<\alpha,\beta_2>_{n-2}\cap<\alpha,\beta_3>_{n-2})=4$, otherwise $\beta_1=\beta_2=\beta_3$, but this is impossible. 
\end{proof}

\section{Normal Bundle of Rational Curves in $\bbP^{n-k}$, for $\frac{n-1}{3}\leq k\leq n-3$}

\begin{obs}\label{lemmacodimr2}  We always have:
$$k\leq \rank
N^L_{n,k}=n-1-h^0(N^{\vee}_{\pi_L(C_n),\bbP^{n-k}}(n+2))\leq n-1.$$
\end{obs}
\begin{propo}
$N_{\pi_L(C_n);\bbP^{n-k}}\isom
\sO(n+2)^{n-1-\rank(N^L_{n,k})}\oplus\mathcal{F}$, where
$\mathcal{F}$ is a vector bundle on $\bbP^1$ of $\rank(N^L_{n,k})-k$
on $\bbP^1$ and $\deg(\mathcal{F}^{\vee}(n+2))=-2k$. We have
$\mathcal{F}\isom\bigoplus^{\rank(N^L_{n,k})-k}_{i=0}\sO(l_i)$ with
$l_i\geq n+3$.

If $2(n-k)\geq 2k$ we have two possibilities:
\begin{enumerate}
\item $N_{\pi_L(C_n);\bbP^{n-k}}\isom \sO(n+2)^{r-1}\oplus\sO(n+3)^{2(n-k-r)-2k}\oplus\mathcal{F}' $ with $\rank(\mathcal{F}')=n-r-2(n-k-r)+k$ and $\deg({\mathcal{F}'}^{\vee}(n+2))=-2k+2(2(n-k-r)-2k)$ if and only if $\rank(N^L_{n,k})=n-r$ and $2(n-k-r)\geq 2k$  for $1\leq r\leq n-k-2$;

\item  $N_{\pi_L(C_n);\bbP^{n-k}}\isom \sO(n+2)^{n-k-2}\oplus \sO(n+2+2k) $ if and only if $\rank(N^L_{n,k})=k+1$.
\end{enumerate}
However the last one is true also for $2n-2k<2k$.
\end{propo}

We can rephrase the above proposition as:
\begin{propo}
If $2(n-k)\geq 2k$ we have two possibilities:
\begin{enumerate}
\item $\pi_{L}(C_n)\in N^n_{n-k}((n+2)^{r-1},(n+3)^{2(n-k-r)-k},spt(\mathcal{F}') )$, where $spt(\mathcal{F}')$ is the splitting type of $\mathcal{F}'$ with $\rank(\mathcal{F}')=n-r-2(n-k-r)$ and $\deg({\mathcal{F}'}^{\vee}(n+2))=-2k+2(2(n-k-r)-2k)$ if and only if $L\in V({N^L_{n,k}})^{n-r}$  and $2(n-k-r)\geq 2k$  for $1\leq r\leq n-k-2$;

\item  $\pi_{L}(C_n)\in N^n_{n-k}((n+2)^{n-k-1},(n+2+k)) $ if and only if $L\in V({N^L_{n,k}})^{k+1}$.

\end{enumerate}
However the last one is true also for $2n-2k<2k$.
\end{propo}

\begin{lem}\label{GADsimilarlemmaN}
$\rank N^L_{n,k}\leq n-2$ if and only if the forms $f_i$ of degree $n$ corresponding to the points $p_i$ generating $L$ can be represented by the similar GAD, i.e. :
$$
f_i=G_{i_1}L^{n-g_1+1}_1+...+G_{i_m}L^{n-g_m+1}_m.
$$
\end{lem}
\begin{proof}
\begin{itemize}
\item[$\Rightarrow$] If $\rank N^L_{n,k}\leq n-2$, then there exists at least an element $\phi\in T_{n-2}$ such that for all forms $f_i$ corresponding to the points $p_i$ generating $L$ we have $\phi\circ f_i=0$. So we can consider the primary decomposition of $\phi=\prod^{m}_{i=1}(\phi_i)^{g_i}$, with $\phi_i\in T_1$ and $\sum_i g_i=n-2$, so every $f_i$ can be represented by the similar GAD, i.e. :
$$
f_i=G_{i_1}L^{n-g_1+1}_1+...+G_{i_m}L^{n-g_m+1}_m,
$$
where $(L_j)^{\perp}=\phi_j$ for all $j=1,..,m$ and $G_{i_j}\in S_{g_j-1}$ for all $i=1,...,k$ and $j=1,...,m$.
\item[$\Leftarrow$] On the other hand if every $f_i$ can be represented by the similar GAD, i.e. :
$$
f_i=G_{i_1}L^{n-g_1+1}_1+...+G_{i_m}L^{n-g_m+1}_m,
$$
then we can consider $\phi=\prod^{m}_{i=1}((L_i)^{\perp})^{g_i}$. By definition of GAD representation we have $\phi\circ f_i=0$ for all $i=1,...,k$, so $\phi\in \ker N^L_{n,k}$ and $\rank N^L_{n,k}\leq n-2$.
\end{itemize}
\end{proof}

\begin{obs}
In particular we can observe that if $L$ belong to a $(n-2)-$secant $\bbP^{n-3}$ generated by $q_1,...,q_{n-2}$, then there exists an element $\phi\in H^0(\sO_{\pi_L(C_n)}^{n-1})\isom S^{n-2}V^{\vee}=T_{n-2}$ such that $\phi\in \ker(N^L_{n,k})=\bigcap_i\ker(Cat_{f_i}(2,n-2))$, in fact we can take $\phi= \prod^{n-2}_{i=1} L^{\perp}_{q_i}$.

We can compute the codimension of the variety of every $\bbP^{k-1}$
which belongs to some $(n-2)-$secant $\bbP^{n-3}$  by constructing an
incidence variety:
$$
I_S=\{(L,\pi):L\in Gr(\bbP^{k-1},\bbP^n),\pi\in S, L\subset S\},
$$
where $S$ is the set of all $(n-2)$-secant $\bbP^{n-3}$ to $C_n$. In
the usual way we can compute the codimension of the image of this
incidence variety in $Gr(\bbP^{k-1},\bbP^n)$. We will indicate with
$\phi_1$ and $\phi_2$ the natural projections:
$$
\xymatrix{
& I_S \ar[dl]_{\phi_1} \ar[dr]^{\phi_2}&\\
Gr(\bbP^{k-1},\bbP^n) & & S,
}
$$
so the codimension in $Gr(\bbP^{k-1},\bbP^n)$ of $\phi_1(I_S)$ is equal to $\dim Gr(\bbP^{k-1},\bbP^n)-\dim S-\dim \phi^{-1}_2(S)=k(n+1-k)-n+2-k(n-2-k)$. The above calculation is effective thanks to the result of Chiantini and Ciliberto on the non-defectivity of the Grassmannians of secant varieties of curves (see \cite{chiantiniciliberto}). We have that this variety has codimension  $3k-n+2$ which is the codimension expected as determinantal variety.
\end{obs}
In general we can prove:
\begin{lem}
If the centre of projection $L\isom\bbP^{k-1}$ belongs to some $(n-1-r)-$secant $\bbP^{n-r-2}$ to the rational normal curve $C_n$ in $\bbP^n$, then we have $\rank N^L_{n,k}\leq n-r$ for $1\leq r< n-k-1$.
\end{lem}
\begin{proof}
If $L\isom\bbP^{k-1}$ belongs to some $(n-r-1)-$secant
$\bbP^{n-r-2}$, then there exist $k$ points $p_1,...,p_k\in L$ which
generate $L$ and the corresponding binary forms $f_i$ are generated
by two forms $Ann(f_i)=(\alpha,\beta_i)$ where $\deg(\alpha)=n-r-1$,
$\alpha$ has only simple roots and $\deg(\beta_i)=r+3$ without
common zeros with $\alpha$.  We have $\dim(\alpha)_{n-2}=r$  and
$\dim \bigcap_i(\alpha,\beta_i)_{n-2}\geq r$, so $\rank
N^L_{n,k}=n-1-\dim \bigcap_i(\alpha,\beta_i)_{n-2}\leq n-r-1$.
\end{proof}
\begin{lem}
The codimension in $Gr(\bbP^{k-1},\bbP^n)$ of the variety of all $L\isom \bbP^{k-1}$ in $\bbP^n$ belonging to some $(n-r-1)$-secant $\bbP^{n-r-2}$ to the rational normal curve in $\bbP^n$  is $2k+kr-n+r+1$.
\end{lem}
\begin{proof}
Infact we can consider the incidence variety $I_S=\{(L,\pi):L\in Gr(\bbP^{k-1},\bbP^n),\pi\in S, L\subset S\}$ where $S$ is the set of all $(n-r-1)$-secant $\bbP^{n-r-2}$ to the rational normal curve in $\bbP^n$. In the usual way we can compute the codimension of the image of this incidence variety in $Gr(\bbP^{k-1},\bbP^n)$. The above calculation is effective thanks to the result of Chiantini and Ciliberto on the non-defectivity of the Grassmannians of secant varieties of curves (see \cite{chiantiniciliberto}).
\end{proof}
In particular for $r=n-k-1$ we have:
\begin{teo}\label{mainresult}
If the centre of projection $L\isom\bbP^{k-1}$ lies in to some $(k+2)-$secant  $\bbP^{k+1}$ to the rational normal curve $C_n$ in $\bbP^n$, then we have:
$$
N_{\pi_L(C_n);\bbP^{n-k}}\isom \sO(n+2)^{n-k-2}\oplus \sO(n+1+2k) .
$$
\end{teo}

\begin{obs}\label{lemmacodimr}
$k+1\leq \rank N^L_{n,k}=n-1-h^0(N^{\vee}_{\pi_{L}(C_n),\bbP^{n-k}}(n+2))\leq 3k$
\end{obs}
\begin{propo}
$N_{\pi_{L}(C_n);\bbP^{n-k}}\isom \sO(n+2)^{n-1-\rank(N^L_{n,k})}\oplus\mathcal{F}$, with $\mathcal{F}$ a vector bundle of rank $\rank(N^L_{n,k})-k$ on $\bbP^1$ and $\deg(\mathcal{F}^{\vee}(n+2))=-2k$ such that $\mathcal{F}\isom \bigoplus^{\rank(N^L_{n,k})-k}_i \sO(l_i)$ with $l_i\geq n+3$.

In this case we have three possibilities:
\begin{enumerate}
\item $\mathcal{F}\isom \sO(n+3)^{2k} $ if and only if $\rank(N^L_{n,2})=3k$;
\item  $\mathcal{F}\isom \sO(n+3)^{2k-2}\oplus\sO(n+4) $ if and only if $\rank(N^L_{n,k})=3k-1$;
\item  $\mathcal{F}\isom \sO(n+3)^{2k-2r}\oplus\mathcal{F}' $  with $\rank(\mathcal{F}')=r$ and $\deg({\mathcal{F}'}^{\vee}(n+2))=-2k$ if and only if $\rank(N^L_{n,k})=3k-r$ with $1<
r\leq 2k-1$;

\item  $\mathcal{F}\isom \sO(n+2+2k) $ if and only if $\rank(N^L_{n,k})=k+1$.
\end{enumerate}
\end{propo}

\section{Normal Bundle of Rational Curves in $\bbP^{n-k}$, for $k<\frac{n-1}{3}$}

\begin{obs}
If $L\isom \bbP^{k-1}$, as centre of projection, belongs to a $(k+1)-$secant $\bbP^k$, then there exist  $k$ binary forms of degree $n$ correspond to points $q_1,..,q_k\in L$ such that $Ann(f_1)=(\alpha,\beta_1),...,Ann(f_k)=(\alpha,\beta_r)$ with $\alpha$ has only simple roots  and $\deg(\alpha)=k+1$ and $\beta_1\neq...\neq\beta_k,\;\; \deg(\beta_1)=...=\deg(\beta_k)=n-k+1$, where $f_i$ is the binary form correspond to the point $q_i$. So $\dim<\alpha>_{n-2}=n-k-2$ and $\dim<\beta_i>_{n-2}=k$, therefore $n-k-1\leq \rank(N^L_{n,k})=n-1-\dim(<\alpha,\beta_1>_{n-2}\cap...\cap<\alpha,\beta_k>_{n-2})\leq k+1$.
\end{obs}

\begin{obs}
We can observe that if $L$ belongs to a $(n-2)-$secant  $\bbP^{n-3}$ generated by $q_1,...,q_{n-1}$, then there exists an element $\phi\in H^0(\sO_{\pi_L(C_n)}^n)\isom S^{n-1}V^{\vee}=T_{n-1}$ such that $\phi\in \ker(N^L_{n,k})=\bigcap_i\ker(Cat_{f_i}(2,n-2))$, in fact we can always take $\phi= \prod^{n}_{i=1} L^{\perp}_{q_i}$, since $\dim \ker N^L_{n,k}\geq 1$.

Unfortunately this condition is empty for $3k<n-2$, in fact we can compute the codimension of the variety of every $\bbP^{k-1}$ which belong to some (n-2)-secant $\bbP^{n-3}$ constructing an incidence variety:
$$
I_S=\{(L,\pi):L\in Gr(\bbP^{k-1},\bbP^n),\pi\in S, L\subset S\},
$$
where $S$ is the set of all (n-2)-secant $\bbP^{n-3}$ to $C_n$. In the usual way we can compute the codimension of the image of this incidence variety in $\bbGr(k-1,n)$. We will indicated with $\phi_1$ and $\phi_2$ the natural projections:
$$
\xymatrix{
& I_S \ar[dl]_{\phi_1} \ar[dr]^{\phi_2}&\\
Gr(\bbP^{k-1},\bbP^n) & & S,
}
$$
so the codimension in $Gr(\bbP^{k-1},\bbP^n)$ of $\phi_1(I_S)$ is equal to $\dim Gr(\bbP^{k-1},\bbP^n)-\dim S-\dim \phi^{-1}_2(S)=k(n+1-k)-n+2-k(n-2-k)$. That calculation is effective thanks to the result of Chiantini and Ciliberto on the non-defectivity of the Grassmannians of secant varieties of curves (see \cite{chiantiniciliberto}). We have that this variety has codimension  $3k-n+1$, but we are in the hypothesis $3k< n-1$, so $3k-n+1<0$.

For $3k=n-1$ the condition gives $\codim =0$, so it is verified for all $L$.
\end{obs}

\section{Restricted Tangent Bundle}
For the restricted tangent bundle we obtain similar results as for
the normal bundle.

\begin{lem}
Let $n\leq 2k$. The codimension in $Gr(\bbP^{k-1},\bbP^n)$ of the variety of  $L\isom \bbP^{k-1}$ in $\bbP^n$ belonging to some $(n-r)$-secant $\bbP^{n-r-1}$ to the rational normal curve in $\bbP^n$  is $k-n+r+kr$.
\end{lem}
In particular for $r=n-k-1$ we have:
\begin{teo}
Let $n\leq 2k$. If the centre of projection $L\isom\bbP^{k-1}$ belongs to some $(k+2)-$secant $\bbP^{k+1}$  to the rational normal curve $C_n$ in $\bbP^n$, then we have:
$$
T\bbP^{n-k}|_{\pi_{k}(C_n)}\isom \sO(n+1)^{n-k-1}\oplus\sO(n+1+k).
$$

\end{teo}

\begin{teo}\label{mainresultTG}
Let $2k\leq n$. If the centre of projection $L\isom\bbP^{k-1}$ belongs to some $\bbP^{k}$ $(k+1)-$secant to the rational normal curve $C_n$ in $\bbP^n$, then we have:
$$
T\bbP^{n-k}|_{\pi_{k}(C_n)}\isom \sO(n+1)^{n-k-1}\oplus\sO(n+1+k).
$$
\end{teo}
\begin{obs}
By Theorem 1 in \cite{Ramella1} (see also \cite{verdier})  $T^n_{n-k}((n+1)^{n-k-1},n+1+k)$ is an irreducible variety of $\codim(T^n_{n-k}((n+1)^{n-k-1},n+1+k))=(k-1)(n-k-1)\leq kn-k^2-k-1$.
\end{obs}
\begin{cor}
Let $2k<n$. The variety of linear spaces $L\isom\bbP^{k-1}$ such
that, $\pi_L(C_n) \subset \bbP^{n-k}$ has the restricted tangent
bundle
$T\bbP^{n-k}_{\pi_{k}(C_n),\bbP^{n-k}}\isom\sO(n+1)^{n-k-1}\oplus\sO(n+1+k)
$ has an irreducible subvariety of codimension $kn-k^2-k-1$ in
$Gr(\bbP^{k-1},\bbP^n)$ formed by the linear spaces $L$ belonging to
some $(k+1)-$secant $\bbP^k$.
\end{cor}
\section*{Acknowledgements}
This paper is part of my PhD thesis. I am very
grateful to my advisor Professor Giorgio Ottaviani for the patience with
which he followed this work very closely.

\newpage

\bibliography{BIBPHDTHESIS}
\bibliographystyle{plain}

\end{document}